\documentclass[12pt]{article}

\usepackage{amsmath}
\usepackage{mathtools}
\usepackage{amsfonts}
\usepackage{amssymb}
\usepackage{graphicx}
\usepackage{hyperref}
\usepackage{amsthm}
\usepackage{fullpage}

\newtheorem{theorem}{Theorem}
\newtheorem{lemma}[theorem]{Lemma}

\newcommand{\RR}{\mathbb{R}}
\newcommand{\NN}{\mathbb{N}}

\newcommand{\expectation}{\operatorname{\mathbb{E}}}
\newcommand{\suchthat}{\mathrel{:}}
\newcommand{\abs}[1]{\lvert#1\rvert}
\newcommand{\vol}{\operatorname{vol}}
\newcommand{\conv}{\operatorname{conv}}

\newcommand{\interior}{\operatorname{int}}
\newcommand{\norm}[1]{{\lVert#1\rVert}}
\newcommand{\norms}[1]{{\lVert#1\rVert}^2}
\newcommand{\inner}[2]{#1\cdot #2}

\def\final{0}  
\def\final{1}  

\ifnum\final=0  
\newcommand{\lnote}[1]{[{\small Luis: \bf #1}]\marginpar{*}}
\newcommand{\anonnote}[1]{[{\small anon: \bf #1}]\marginpar{*}}
\newcommand{\sidecomment}[1]{\marginpar{\tiny #1}}
\newcommand{\details}[1]{[[#1]]}
\else 
\newcommand{\lnote}[1]{}
\newcommand{\anonnote}[1]{}
\newcommand{\sidecomment}[1]{}
\newcommand{\details}[1]{}
\fi  
\usepackage{color}
\usepackage{bbm}

\newcommand{\linspan}{\operatorname{span}}
\newcommand{\ones}{\ensuremath{\mathbbm{1}}}
\newcommand{\e}{\expectation}
\newcommand{\aff}{\operatorname{aff}}
\newcommand{\ud}{\mathrm{d}}
\newcommand{\distD}{\mathcal{D}}
\newcommand{\email}[1]{\href{mailto:#1}{\texttt{#1}}}
\newcommand{\rec}{\operatorname{rec}}
\newcommand{\dom}{\operatorname{dom}}
\newcommand{\evaluatedat}[2]{\left.{#1}\right\rvert_{#2}}

\title{A simplicial polytope that maximizes the isotropic constant must be a simplex}

\author{Luis Rademacher\\
Computer Science and Engineering\\
Ohio State University\\
\email{lrademac@cse.ohio-state.edu}}

\date{}

\begin{document}
\maketitle

\begin{abstract}
The isotropic constant $L_K$ is an affine-invariant measure of the spread of a convex body $K$. For a $d$-dimensional convex body $K$, $L_K$ can be defined by $L_K^{2d} = \det(A(K))/(\mathrm{vol}(K))^2$, where $A(K)$ is the covariance matrix of the uniform distribution on $K$. It is an outstanding open problem to find a tight asymptotic upper bound of the isotropic constant as a function of the dimension. It has been conjectured that there is a universal constant upper bound. The conjecture is known to be true for several families of bodies, in particular, highly symmetric bodies such as bodies having an unconditional basis. It is also known that maximizers cannot be smooth. 

In this work we study the gap between smooth bodies and highly symmetric bodies by showing progress towards reducing to a highly symmetric case among non-smooth bodies. More precisely, we study the set of maximizers among simplicial polytopes and we show that if a simplicial polytope $K$ is a maximizer of the isotropic constant among $d$-dimensional convex bodies, then when $K$ is put in isotropic position it is symmetric around any hyperplane spanned by a $(d-2)$-dimensional face and the origin. By a result of Campi, Colesanti and Gronchi, this implies that a simplicial polytope that maximizes the isotropic constant must be a simplex.

%
\end{abstract}

\section{Introduction}

Let a $d$-dimensional \emph{convex body} be a compact convex set in $\RR^d$ with non-empty interior. 
For a $d$-dimensional convex body $K$, let $A(K)$ be the covariance matrix of the uniform distribution on $K$, that is, for $X$ random in $K$, let $\mu_K = \e (X)$ and let $A(K) = \e_K \bigl((X-\mu_K)(X-\mu_K)^T\bigr)$. 
The \emph{isotropic constant} $L_K$ of $K$ can be defined by $L_K^{2d} = \det\bigl(A(K)\bigr)/(\mathrm{vol}(K))^2$. 
An outstanding open problem in asymptotic geometric analysis, the \emph{slicing problem}, is to find a tight asymptotic upper bound of the isotropic constant as a function of the dimension only. 
The \emph{slicing conjecture} (also known as the \emph{hyperplane conjecture}) states that there is a universal constant upper bound of the isotropic constant.
The slicing problem seems to be mentioned for the first time by Bourgain \cite{Bourgain1986} and some equivalent formulations are discussed by Ball in \cite{Ball1988}. Milman and Pajor \cite{MilmanPajor} studied the problem systematically. It is one of the outstanding open problems in convex geometry; among the reasons are its connections with classical problems in convexity, like the Busemann-Petty problem and Sylvester's problem \cite{MilmanPajor, Giannopoulos, Giannopoulos2}.

The isotropic constant is invariant under affine transformations. About the asymptotic behavior, it is known that $L_K$ is bounded below by a universal constant and this is tight \cite{Bourgain1986}. Bourgain also showed that $L_K \leq O( \sqrt[4]{d} \log d)$. The best known upper bound at this time is by Klartag, who showed that $L_K \leq O(\sqrt[4]{d})$ \cite{Klartag}. 

A natural question is to understand bodies that maximize $L_K$ in every dimension. We review now some results that give a partial understanding of the nature of maximizers. There are subfamilies of all convex bodies where universal constant upper bounds to $L_K$ are known: 
\begin{itemize}
\item zonoids, that is, limits of zonotopes (which are Minkowski sums of segments),

\item bodies having an unconditional basis, that is, convex bodies $K$ such that $(x_1, \dotsc, x_d) \in K \iff (\abs{x_1}, \dotsc, \abs{x_d}) \in K$. \cite{Bourgain1986, MilmanPajor}
\end{itemize}
Also, \cite{MR1763244} showed (using RS-movements) that if $K$ has a non-empty subset of its boundary of class $C^2$ with positive principal curvatures, then it cannot be a maximizer of $L_K$. It also showed using Blaschke's Sch\"uttelung (shakedown) that if $K = \{ (x,y) \in \RR^{d-1} \times \RR \suchthat x \in \pi_H(K), -f(x) \leq y \leq f(x)\}$ (i.e. $K$ has a hyperplane of symmetry) and $f$ is not affine, then $K$ is not a maximizer of the isotropic constant (see Theorem \ref{thm:ccg} in this paper for a formal statement).

The results above tell us that for bodies having certain symmetries or having a smooth boundary we have some understanding of the behavior of the isotropic constant. This leaves unexplored bodies that are not smooth and have no symmetries. Can we argue that maximizing bodies must have certain symmetries?

A polytope is \emph{simplicial} if every facet is a simplex. The first result of this paper is the following necessary condition on simplicial polytopes that are local extrema of the isotropic constant, Theorem \ref{thm:symmetry}.
We first recall a few definitions and basic facts.
In the context of this paper, a $d$-dimensional convex body is a local extremum of a functional iff the body is a local maximum or a local minimum of the functional in the space of all convex bodies with respect to the Hausdorff topology. A $d$-dimensional polytope is called isohedral or a isohedron iff it is facet-transitive, that is, for any pair of facets $F_1, F_2$ there is an element of the (full) symmetry group of the polytope that maps $F_1$ to $F_2$. (The full symmetry group of a polytope is the group (under composition) of all isometries of $\RR^d$ that map the polytope to itself.) A $d$-dimensional convex body $K$ is \emph{isotropic} iff $\mu_K = 0$ and $A(K) = I$. Given any $d$-dimensional convex body, there is an affine transformation that maps it to an isotropic convex body. As the isotropic constant is invariant under affine transformations, it is enough to study it over isotropic convex bodies.
\begin{theorem}\label{thm:symmetry}
Let $P$ be a $d$-dimensional isotropic simplicial polytope that is a local extremum of $P \mapsto L_P$. Let $H$ be any hyperplane spanned by a $(d-2)$-dimensional face of $P$ and the origin. Then $P$ is symmetric around $H$.
In particular, $P$ is isohedral and all facets are congruent.
\end{theorem}

The second and main result of this paper is Theorem \ref{thm:main}, a corollary of Theorem \ref{thm:symmetry} and the result from \cite{MR1763244} mentioned before (stated in this paper as Theorem \ref{thm:ccg}):
\begin{theorem}\label{thm:main}
Let $P$ be a $d$-dimensional simplicial polytope that is a maximizer of $P \mapsto L_P$. Then $P$ is a simplex.
\end{theorem}

It is worth mentioning a parallel situation for Mahler's problem: a similar gap between smooth bodies and symmetric bodies. For simplicity, we only discuss the \emph{symmetric} Mahler problem: Given a centrally symmetric convex body $K$, define the \emph{volume product} as $f(K) = \vol(K) \vol(K^\circ)$, where $K^\circ$ is the polar body of $K$, given by
\[
K^\circ := \{ x \in \RR^d \suchthat x \cdot y \leq 1 \; \forall y \in K\}.
\]
The problem is to determine the minimizers of the volume product. The cube is conjectured to be a minimizer. This conjecture has been verified among unconditional bodies \cite{MR670798, MR876398} (which correspond to symmetries around coordinate hyperplanes) and this result has been extended to more general symmetries \cite{MR3038713}. On the other hand, \cite{MR2874925} (improving earlier work, \cite{MR2547812}), showed that if a convex body $K$ has a boundary point with positive generalized Gauss curvature, then $K$ cannot minimize the volume product.



The rest of the paper is devoted to the proofs of Theorems \ref{thm:symmetry} and \ref{thm:main}.

\section{Preliminaries}


For $n \in \NN$, let $[n]$ denote the set $\{1, \dotsc, n \}$.

A polytope in $\RR^d$ is the convex hull of a finite number of points. For a polytope $P$ with non-empty interior, the unique irredundant $H$-representation of $P$ is the unique set of closed halfspaces $\{ H_i \suchthat i=1, \dotsc, n\}$ such that $P = \cap_{i=1}^n H_i$. See \cite[Chapters 0,1,2]{MR1311028} for background.

We use a result from \cite{MR1763244} that rules out many bodies having a hyperplane of symmetry as potential maximizers of the isotropic constant. We state the result, Theorem \ref{thm:ccg}, after a few definitions.
Let $K$ be a $d$-dimensional convex body. 
Let $H \subseteq \RR^d$ be a hyperplane through the origin. 
Let $\pi_H:\RR^d \to \RR^d$ be the orthogonal projection onto $H$. Let $f_H, g_H : \pi_H(K) \to \RR$ be convex functions such that 
\begin{equation}\label{equ:representation}
K = \{ (x,y) \in \RR^{d-1} \times \RR \suchthat x \in \pi_H(K), f_H(x) \leq y \leq -g_H(x)\}.
\end{equation}

\begin{theorem}[\cite{MR1763244}]\label{thm:ccg}
If a $d$-dimensional convex body $K$ is symmetric with respect to a hyperplane through the origin $H$ and the corresponding functions $f_H$, $g_H$ are not affine, then $K$ is not a maximizer of the isotropic constant.
\end{theorem}
(Note that in this case we have $f_H = g_H$.)
\begin{proof}
Immediate from \cite[Theorem 3.2]{MR1763244},  \cite[Theorem 3.6]{MR1763244} and the fact that 
\[
L_K^{2d} = \frac{ \det A(K)}{ \vol(K)^2 } = \frac{d!}{d+1} M_2(K;d+1)
\]
where $M_2(K;d+1)$ is the second moment of the volume of a random simplex in $K$, normalized so that it is affinely invariant. That is,
\[
M_2(K;d+1) = \frac{1}{\vol(K)^{d+3}} \int_{x_1 \in K} \dotsb \int_{x_{d+1}\in K} \bigr(\vol \conv (x_1, \dotsc, x_{d+1})\bigr)^2 \, \ud x_{d+1}\dotsm \ud x_{1}.
\]
\end{proof}
\section{Derivative with respect to hinging}

In this section we compute the derivative of the isotropic constant of a polytope with respect to hinging of one of its facets. The first step (the next lemma) is to show conditions under which integrals of functions over polytopes are differentiable as a facet hinges. This is slightly non-trivial; similar differentiability issues have been discussed before \cite[Lemma 2.4]{MR2132704}.
\begin{lemma}\label{lem:differentiability}
Let $P \subseteq \RR^d$ be a closed polyhedron with non-empty interior such that the origin is on the boundary of $P$. Let $Z$ be an $(d-2)$-dimensional subspace that does not intersect the interior of $P$. Let $\theta:\RR^d \to [-\pi, \pi)$ be an angle in a system of polar coordinates (see Figure \ref{fig:differentiability}) on $Z^\perp$ so that 
\begin{enumerate}
\item There are $a,b \in (-\pi, \pi)$ satisfying $\theta(P) = [a,b]$ and $a<0<b$.
\item $P \cap \{x \suchthat a \leq \theta(x) \leq 0\}$ is bounded.
\end{enumerate}
Let $f: P \to \RR$ be a continuous function. Let $K_t = \{x \in P \suchthat \theta(x) \in [-\pi, t]\}$. Then $g(t) = \int_{K_t} f(x) \ud x $ is differentiable at $0$.
\end{lemma}

\begin{figure}
\begin{center}
\def\svgwidth{.5\columnwidth}
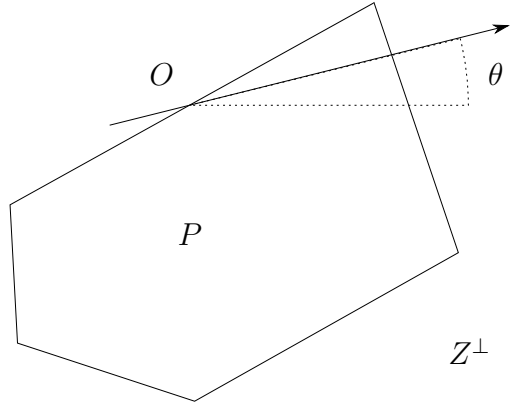
\caption{Coordinates for the proof of Lemma \ref{lem:differentiability}.}\label{fig:differentiability}
\end{center}
\end{figure}

\begin{proof}
Let $r:S^{d-1} \to [0, \infty]$ be the radial function of $P$. That is, $r(u) = \max \{ a \in [0, \infty] \suchthat a u \in P\}$.

Write $g(t)$ in hyperspherical coordinates:
\begin{equation}\label{equ:hyperspherical}
g(t) = 
\int_{\theta=-\pi}^t 
\int_{\phi_{d-2}=0}^\pi \dotsm \int_{\phi_{1}=0}^\pi \int_{\rho=0}^{r(v(\phi, \theta))} 
f\bigl(\rho v(\phi, \theta)\bigr)
J(\rho, \phi, \theta)
\ud \rho \ud \phi_1 \ud \phi_2 \dotsm \ud \phi_{d-2} 
\ud\theta,
\end{equation}
where $\phi = (\phi_1, \dotsc, \phi_{d-1})$, $J(\rho, \phi, \theta)$ is the volume element, given by
\[
J(\rho, \phi, \theta) = \rho^{d-1} \sin^{d-2}(\phi_1) \sin^{d-3}(\phi_2) \dotsm \sin(\phi_{d-2}),
\]
and $v(\phi, \theta)$ is the change of variable that gives a point on $S^{d-1}$ given the angles, that is,
$v(\phi, \theta)_1 = \cos(\phi_i)$, $v(\phi, \theta)_2 = \sin(\phi_1) \cos(\phi_2)$, \ldots, $v(\phi, \theta)_{d-1} = \sin(\phi_1) \dotsm \sin(\phi_{d-2}) \cos(\theta)$, $v(\phi, \theta)_{d} = \sin(\phi_1) \dotsm \sin(\phi_{d-2}) \sin(\theta)$.

Let $h(\theta)$ be the integrand of the outermost integral in \eqref{equ:hyperspherical}. To establish that $g$ is differentiable at 0, it is enough to show that $h(\theta)$ is continuous in a neighborhood of 0.

To simplify the argument we can go to a fixed domain of integration for $h$ by the change of variable $\lambda = \rho/r\bigl(v(\phi, \theta)\bigr)$ so that
\begin{align*}
h(\theta) = 
\int\limits_{\phi_{d-2}=0}^\pi \dotsm \int\limits_{\phi_{1}=0}^\pi \int\limits_{\lambda=0}^{1} 
&f\bigl(\lambda r(v(\phi, \theta)) v(\phi, \theta) \bigr) \cdot \\
&J\bigl(\lambda r(v(\phi, \theta)), \phi, \theta\bigr) \,
r\bigl(v(\phi, \theta)\bigr) \,
\ud \lambda \ud \phi_1 \ud \phi_2 \dotsm \ud \phi_{d-2} .
\end{align*}

We first prove that $h(\theta)$ is finite in a neighborhood of 0. To see this, it is enough to show that $r\bigl(v(\phi, \theta)\bigr)$ is bounded for $\theta$ in a neighborhood of 0 and any $\phi \in [0,\pi]^{d-2}$.

The behavior of $r(v)$ on $S^{d-1}$ is as follows:  $r(v)=\infty$ iff $v$ is in the recession cone of $P$ (denoted $\rec(P)$ and given by $\rec(P) = \{ x \in \RR^d \suchthat  x + P \subseteq P\}$, see \cite[Section 1.4]{MR1216521} or \cite[Section 8]{Rockafellar70}). \details{$0 \in P$ implies $\rec(P) \subseteq P$ and this gives ($\Leftarrow$). 
For the other direction, if $r(v)=\infty$, then $Mv \in P$ for all $M \geq 0$. 
Take an arbitrary point $y \in P$. 
It is enough to show $v + y \in P$. To see, use the convexity of $P$ to get $\lambda M v + (1-\lambda) y \in P$ for all $M\geq0$ and $\lambda =1/M$. That is, $v + (1-\frac{1}{M}) y \in P$. Make $M \goesto \infty$ and use the fact that $P$ is closed to conclude that $v + y \in P$.} Assumption 2 implies that $r(v)$ is finite for $\theta=0$ and any $\phi$. As $P$ is a closed convex body, its recession cone is closed. On $S^{d-1}$, $\{ \theta=0 \} \cap S^{d-1}$ is compact and $\rec(P) \cap S^{d-1}$ is also compact, and they are disjoint. Therefore, there is an open neighborhood $A \subseteq S^{d-1}$ containing $\{ \theta=0 \} \cap S^{d-1}$ such that $A \cap \rec(P) =\emptyset$. Moreover, $r$ is continuous in $A$ 
 (see footnote\footnote{Let $D \subseteq S^{d-1}$ be the set where $r(v) \in (0,\infty)$. Then $r(v)$ is continuous in $\interior D$. To see this, note that $r(v) = 1/f(v)$, where $f(v) = \inf \{\lambda>0 \suchthat v \in \lambda P\}$ is the (slightly generalized) Minkowski functional of $P$. We have $f:\RR^d \to [0,\infty]$ and $f$ is convex. Therefore $f$ is continuous in $\interior \dom f$ \cite[Theorem 2.35]{MR1491362}, where $\dom f$ is the effective domain of $f$, that is, $\dom f = \{x \in \RR^d: f(x) < \infty. \}$. Clearly $\interior D \subseteq \interior \dom f$.}). 
There is also a smaller compact neighborhood $C$ satisfying $\{ \theta=0 \} \cap S^{d-1} \subseteq C \subseteq A$, so that $r$ is not only continuous but also bounded in $C$. Therefore, for $\theta$ in a neighborhood of $0$ we have that $h(\theta)$ is the integral of a bounded function on a compact domain of integration. This implies that $h(\theta)$ is finite in a neighborhood of 0. 

We now establish that $h(\theta)$ is continuous in the same neighborhood of 0. This follows from Lebesgue's dominated convergence theorem after observing that $h$ is the integral of a continuous integrand over a compact domain, and the integrand is (uniformly) bounded in that domain for $\theta$ in a neighborhood of 0.
\end{proof}

\begin{figure}
\begin{center}
\def\svgwidth{.5\columnwidth}
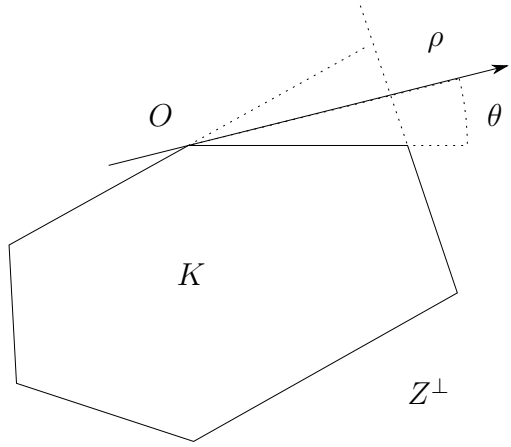
\caption{Cylindrical coordinates for the proof of Lemma \ref{lem:derivative}.}\label{fig:cylindrical}
\end{center}
\end{figure}

\begin{lemma}[derivative with respect to hinging]\label{lem:derivative}
Let $K \subseteq \RR^d$ be an isotropic simplicial polytope with unique irredundant $H$-representation $\{ H_i \suchthat i=1, \dotsc, m\}$ so that $K = \cap_{i=1}^m H_i$. Fix a facet $i'$ and let $P$ be the polyhedron obtained by relaxing $H_{i'}$, that is:
\[
P = \bigcap_{i \neq i'} H_i.
\]
Let $v_1, \dotsc, v_d$ be the vertices of facet $i'$. 
Let $v_{j'}$ be one of those vertices. 
Consider a system of cylindrical coordinates as follows (see Figure \ref{fig:cylindrical}): 
A point $x \in \RR^d$ is parameterized by $(\rho, \theta, z)$. 
Let $Z = \aff \{v_j \suchthat j \neq j'\}$.
The origin of the new system of coordinates is an arbitrary point in $Z \cap K$.
Parameters $\rho \in [0, \infty)$, $\theta \in [-\pi, \pi)$ are polar coordinates in the 2-dimensional plane orthogonal to $Z$.
Parameter $z$ is a vector in $Z$. 
Polar coordinates $\rho, \theta$ are oriented so that points in facet $i'$ have $\theta = 0$ and all of $K$ has $\theta \in (-\pi, 0]$. 
Let 
\begin{align*}
K_t &= \{x \in P \suchthat \theta(x) \in [-\pi, t]\},\\
S_t &= P \cap \{ x \suchthat \theta(x) = t\}.
\end{align*}
Let $\distD$ be the distribution of a random vector $X$ supported on $S_0$ with density proportional to $\rho(X)$.
Then
\[
\left.\frac{d}{dt} L_{K_t}^{2d} \right|_{t=0} = \bigl(\e_{X \leftarrow \distD}(\norms{X}) - d - 2 \bigr)\frac{\evaluatedat{\frac{d}{dt} \vol K_t}{t=0}}{(\vol K)^3}.
\]
\end{lemma}
\begin{proof}
We have $L_{K_t}^{2d} = \det A(K_t) /\vol(K_t)^2$. Given this formula, the differentiability of $t \mapsto L_{K_t}^{2d}$ and of the varied integrals below at $t=0$ follows from Lemma \ref{lem:differentiability}.

We first compute $\left.\frac{d}{dt} \det A(K_t)\right|_{t=0}$. Part of the argument is the same as the proof of \cite[Proposition 15]{MR2891161}: An identical argument shows that $\left.\frac{d}{dt} \det A(K_t)\right|_{t=0} = \left.\frac{d}{dt} \e_{X \in K_t} (\norms{X})\right|_{t=0}$, then we compute this last expression for the current lemma. To make the argument more readable and self-contained, we repeat the relevant parts of the argument from \cite{MR2891161} here.

We have
\begin{align*}
A(K_t) &= \e_{X\in K_t} \bigl((X- \mu_{K_t})(X- \mu_{K_t})^T\bigr) \\
&= \e_{X\in K_t} (X X^T) - \mu_{K_t} \mu_{K_t}^T.
\end{align*}
By isotropy, $\mu_K=0$ and this implies
\begin{equation}\label{equ:dAk}
\begin{aligned}
\evaluatedat{\frac{d}{dt} A(K_t)}{t=0}
    &= \evaluatedat{\frac{d}{dt} \e_{X\in K_t} (X X^T)}{t=0}.
\end{aligned}
\end{equation}
Use the identity
\[
\frac{d}{dM} \det M = \left(M^{-1} \right)^T \det M
\]
to conclude
\begin{align*}
\evaluatedat{\frac{d}{dt} \det A(K_t)}{t=0}
    &= \evaluatedat{\evaluatedat{\frac{d}{dM} \det M}{M=A(K_t)} \cdot \frac{d}{dt} A(K_t)}{t=0} \\
    &= \evaluatedat{\det \bigl(A(K_t)\bigr) \left(A(K_t)^{-1}\right)^T \cdot \frac{d}{dt} A(K_t)}{t=0}
\end{align*}
where the dot ``$\cdot$'' represents the Frobenius inner product of matrices, $M \cdot N = \sum_{ij} M_{ij} N_{ij}$. This, isotropy and \eqref{equ:dAk} give
\begin{align*}
\evaluatedat{\frac{d}{dt} \det A(K_t)}{t=0}
    &= I \cdot \evaluatedat{\frac{d}{dt} \e_{X \in K_t} (X X^T)}{t=0} \\
    &= \evaluatedat{\frac{d}{dt} \e_{X \in K_t} (\norms{X})}{t=0}.
\end{align*}

 
Let $S_{r,t} = P \cap \{ x \suchthat \rho(x) = r, \theta(x) = t\}$. 
We have:
\begin{align*}
\e_{X \in K_t} (\norms{X})
&= \frac{1}{\vol K_t} \int_{-\pi}^t \int_0^\infty \e_{X \in S_{\rho, \theta}} (\norms{X}) \vol_{d-2} (S_{\rho, \theta}) \rho \,\ud\rho \,\ud\theta.
\end{align*}
Differentiating at $t=0$:
\begin{equation}\label{equ:diffnorm}
\begin{aligned}
&\evaluatedat{\frac{d}{dt} \e_{X \in K_t} (\norms{X})}{t=0}
= \\ 
&\evaluatedat{- \frac{\frac{d}{dt} \vol K_t}{\vol K_t} \e_{X \in K_t} (\norms{X}) + 
\frac{1}{\vol K_t} \int_0^\infty \e_{X \in S_{\rho, t}} (\norms{X}) \vol_{d-2} (S_{\rho, t}) \rho \,\ud\rho}{t=0}. 
\end{aligned}
\end{equation}
We have
\begin{equation}\label{equ:volume}
\begin{aligned}
\evaluatedat{\frac{d}{d t} \vol(K_t)}{t=0} 
&= \evaluatedat{\frac{d}{d t} \int_{-\pi}^t \int_0^\infty \vol_{d-2}(S_{\rho, \theta}) \rho \ud \rho \ud \theta}{t=0} \\
&= \int_0^\infty \vol_{d-2}(S_{\rho,0}) \rho \ud \rho,
\end{aligned}
\end{equation}
\begin{equation}\label{equ:norm}
\begin{aligned}
\e_{X \leftarrow \mathcal{D}}(\norms{X}) 
&= \frac{\int_{S_0} \norms{x} \rho(x) \ud x}{\int_{S_0} \rho(x) \ud x} \\
&= \frac{\int_{\rho=0}^\infty \int_{S_{\rho, 0}} \norms{x} \rho \ud x \ud \rho}{\int_{\rho=0}^\infty \vol_{d-2} (S_{\rho,0}) \rho \ud \rho}
\end{aligned}
\end{equation}
and
\begin{equation}\label{equ:norm2}
\begin{aligned}
\e_{X \in S_{\rho, t}}(\norms{X}) = \frac{\int_{S_{\rho, t}} \norms{x} \ud x}{\vol_{d-2}(S_{\rho, t})}.
\end{aligned}
\end{equation}
That is, using \eqref{equ:volume}, \eqref{equ:norm}, \eqref{equ:norm2} we get
\begin{align*}
\int_0^\infty \e_{X \in S_{\rho, 0}}(\norms{X}) \vol_{d-2} (S_{\rho, 0}) \rho \, \ud \rho
&= \int_{\rho=0}^\infty \int_{S_{\rho, 0}} \norms{x} \ud x \rho \, \ud \rho \\
&= \e_{X \leftarrow \mathcal{D}} (\norms{X}) \int_{\rho=0}^\infty \vol_{d-2} (S_{\rho, 0}) \rho \, \ud \rho \\
&= \e_{X \leftarrow \mathcal{D}} (\norms{X}) \evaluatedat{\frac{d}{dt} \vol(K_t) }{t=0}.
\end{align*}
This in \eqref{equ:diffnorm} and isotropy give
\begin{align*}
\evaluatedat{\frac{d}{dt} \e_{X \in K_t} (\norms{X})}{t=0}
&= \frac{\evaluatedat{\frac{d}{dt} \vol K_t}{t=0}}{\vol K} \left[ - d + \e_{X \leftarrow \distD}(\norms{X}) \right].
\end{align*}
This implies,
\[
\left.\frac{d}{dt} \det A(K_t)\right|_{t=0} 
= \bigl(\e_{X \leftarrow \distD}(\norms{X}) - d \bigr)\frac{\evaluatedat{\frac{d}{dt} \vol K_t}{t=0}}{\vol K}.
\]
Thus,
\begin{align*}
\left.\frac{d}{dt} L_{K_t}^{2d} \right|_{t=0} &= \left. \frac{d}{dt} \frac{\det A(K_t)}{(\vol K_t)^2} \right|_{t=0} \\
&= \left. \frac{-2 \det A(K_t)}{(\vol K_t)^3} \frac{d}{dt} \vol K_t + \frac{1}{(\vol K_t)^2} \frac{d}{dt} \det A(K_t) \right|_{t=0} \\
&= \frac{-2}{(\vol K)^3} \evaluatedat{\frac{d}{dt} \vol K_t}{t=0} + \frac{1}{(\vol K)^2} \bigl(\e_{X \leftarrow \distD}(\norms{X}) - d \bigr)\frac{\evaluatedat{\frac{d}{dt} \vol K_t}{t=0}}{\vol K} \\
&= \bigl(\e_{X \leftarrow \distD}(\norms{X}) - d - 2\bigr)\frac{\evaluatedat{\frac{d}{dt} \vol K_t}{t=0}}{(\vol K)^3}.
\end{align*}

\end{proof}

%

\section{Proof of Theorem \ref{thm:symmetry}}

\emph{Proof idea:} For a given facet $F$ of $P$, we compute the derivative of $L_P^{2d}$ as $F$ ``hinges'' around one of its $(d-2)$-dimensional facets. 
To preserve convexity, this is done by hinging the halfspace inducing that facet.
Setting those derivatives to zero gives $d$ non-linear equations on the $d$ vertices of $F$. 
We will show that, if we fix $d-1$ vertices $v_1, \dotsc, v_{d-1}$, the system of equations determines the last vertex uniquely up to reflection around the span of $v_1, \dotsc, v_{d-1}$. 
This implies that the two facets of $P$ sharing $v_1, \dotsc, v_{d-1}$ are reflections of each other around that span. Applying the same argument to all pairs of adjacent facets, we conclude that all of $P$ is invariant under reflection around that span.
\begin{proof}
Let $P$ be a polytope as in the theorem. Let $F$ and $F'$ be two adjacent facets of $P$ with common vertices $\{v_2, \dotsc, v_d\}$ and additional vertices $v_1$ and $v'_1$, respectively. The first order necessary condition from Lemma \ref{lem:foc} implies conditions on $v_1$ and $v_1'$ given $v_2, \dotsc, v_d$. Let $v_1$ be denoted $x$ in those equations, with $v_2, \dotsc, v_d$ as parameters. We get two types of equations: For $k=1$ we get
\begin{align*}
\frac{(d+1)(d+2)^2}{2} = 3 \norms{x} + 2 \sum_{i=2}^d \inner{x}{v_j} + \sum_{2\leq i\leq j\leq d} \inner{v_i}{v_j}.
\end{align*}
For $k = 2, \dotsc, d$ we get
\begin{align*}
\frac{(d+1)(d+2)^2}{2} = \norms{x} + \sum_{i=2}^d \inner{x}{v_j} + \sum_{2\leq i\leq j\leq d} \inner{v_i}{v_j} + \inner{x}{v_k} + \sum_{k \leq j \leq d} \inner{v_k}{v_j} + \norms{v_k}.
\end{align*}
Decompose $x$ into its components in the span of $\{v_2, \dotsc, v_d\}$ and orthogonal to that span. That is, let $x=x_{IN} + x_{OUT}$, with $x_{IN} \in \linspan \{v_2, \dotsc, v_d\}$ and $x_{OUT} \in \linspan \{v_2, \dotsc, v_d\}^\perp$. 
To get that $F$ and $F'$ are reflections of each other around $\linspan \{v_2, \dotsc, v_d\}$, it is enough to show that the equations on $x$ determine $x_{IN}$ uniquely and $\norm{x_{OUT}}$ uniquely. 
In the new variables, the equations on $x$ become:
\begin{equation}\label{equ:one}
\frac{(d+1)(d+2)^2}{2} = 3 \norms{x_{IN}} + 3 \norms{x_{OUT}} + 2 \sum_{i=2}^d \inner{x_{IN}}{v_j} + \sum_{2\leq i\leq j\leq d} \inner{v_i}{v_j}
\end{equation}
and $k = 2, \dotsc, d$ we get
\begin{equation}\label{equ:two}
\frac{(d+1)(d+2)^2}{2} = \norms{x_{IN}} + \norms{x_{OUT}} + \sum_{i=2}^d \inner{x_{IN}}{v_j} + \sum_{2\leq i\leq j\leq d} \inner{v_i}{v_j} + \inner{x_{IN}}{v_k} + \sum_{k \leq j \leq d} \inner{v_k}{v_j} + \norms{v_k}.
\end{equation}
Note that the only non-linear part is $\norms{x_{IN}} + \norms{x_{OUT}} = \norms{x}$. Use \eqref{equ:one} in \eqref{equ:two} to eliminate $\norms{x}$. We get
\[
\norms{x} = \frac{1}{3} \frac{(d+1)(d+2)^2}{2}  - \frac{2}{3} \sum_{i=2}^d \inner{x_{IN}}{v_j} - \frac{1}{3} \sum_{2\leq i\leq j\leq d} \inner{v_i}{v_j}.
\]
and for $k=2, \dotsc, d$,
\begin{equation}\label{equ:three}
\frac{1}{3} \inner{x_{IN}}{\sum_{j=2}^d v_j} + \inner{x_{IN}}{v_k} + \frac{2}{3} \sum_{2\leq i\leq j\leq d} \inner{v_i}{v_j} + \sum_{k \leq j \leq d} \inner{v_k}{v_j} + \norms{v_k} = \frac{2}{3} \frac{(d+1)(d+2)^2}{2}.
\end{equation}
As $x_{IN}$ lies in $\linspan \{v_2, \dotsc, v_d\}$, we can write the linear system of equations \eqref{equ:three} in the basis $\{v_2, \dotsc, v_d\}$, and $x_{IN}$ will be unique as long as the matrix of the linear system is invertible. The matrix of the system in this basis is $\ones \ones^T /3 + I$, which is invertible. Therefore, there is at most one $x_{IN}$. Given $x_{IN}$, \eqref{equ:one} determines $\norm{x_{OUT}}$ uniquely, as claimed.
\end{proof}
\begin{lemma}[First order necessary condition]\label{lem:foc}
Let $P$ be a $d$-dimensional simplicial polytope that is a local extremum of $P \mapsto L_P$. Let $F$ be a facet of $P$ with vertices $v_1, \dotsc, v_d$. Then, for $k=1, \dotsc, d$ we have
\begin{align*}
\sum_{1\leq i \leq j \leq d} (1 + \delta_{ik} + \delta_{jk}) \inner{v_i}{v_j}=\frac{(d+1)(d+2)^2}{2}.
\end{align*}
\end{lemma}
\noindent\emph{Proof idea:} To first order terms and after taking care of preserving convexity, the infinitesimal hinging of a facet $F = \conv(v_1, \dotsc, v_d)$ around one of its facets $F''=\conv(v_2, \dotsc, v_d)$ is the same as adding or removing an infinitesimal layer of mass on the facet with weight proportional to barycentric coordinate $x_1$ (that is, $0$ at $F''$, $1$ at $v_1$ and interpolate affinely in $F$). 
On the other hand,
the derivative of $L_P^{2d}$, when adding infinitesimal mass to an isotropic convex body $P$, is proportional to
\[
\e (\norms{X}) - d - 2
\]
where $X$ is random according to the added mass.  This is show in Lemma \ref{lem:derivative} for our particular case. An explicit computation of $\e (\norms{X})$ completes the argument.
\begin{proof}
Lemma \ref{lem:derivative} implies the necessary condition
\[
\e_{X \leftarrow \distD}(\norms{X}) = d + 2
\]
We now compute $\e_{X \leftarrow \distD}(\norms{X})$ using the known representation of the uniform distribution over a simplex using exponential random variables. Let $Y$ be a random vector, uniformly in the standard simplex $\Delta^{d-1} = \conv\{e_i\}_{i=1}^d$. Let $T$ be an independent random scalar distributed as the sum of $d$ exponential random variables with rate 1. (An exponential random variable with rate 1 has density $e^{-t}$ supported on $t\in [0,\infty)$.) It is known that $Z:= TY$ is a random vector with independent coordinates, each distributed as exponential with rate 1. Let $V$ be the matrix having columns $v_1, \dots, v_d$. Then $VY$ is a uniformly random vector on $\conv \{v_i\}$. 

We have
\[
\e_{X \leftarrow \distD}(\norms{X}) = \frac{\e(\norms{VY} Y_k)}{\e(Y_k)},
\]
where $\e(Y_k) = 1/d$. Also, 
\[
\e(\norms{VY} Y_k ) = \frac{\e(\norms{VZ} Z_k)}{\e(T^3)},
\]
where $\e(T^3) = d(d+1)(d+2)$. Finally, 
\begin{align*}
\e(\norms{VZ} Z_1) 
&= \e\left(\Bigl\lVert\sum_i Z_i v_i\Bigr\rVert^2 Z_1\right) \\
&= \e\Bigl(\sum_{i=1}^d Z_i^2 v_i \cdot v_i Z_1 + 2 \sum_{i<j} Z_i Z_j v_i \cdot v_j Z_1 \Bigr) \\
&= \e(Z_1^3) v_1 \cdot v_1 + \sum_{i=2}^d \e(Z_1 Z_i^2) v_i \cdot v_i \\ 
&\qquad+ 2 \sum_{j=2}^d \e(Z_1^2 Z_j) v_1 \cdot v_j + 2 \sum_{1<i<j\leq d} \e(Z_1 Z_i Z_j) v_i \cdot v_j \\
&= 6 v_1 \cdot v_1 + 2 \sum_{i=2}^d v_i \cdot v_i + 4 \sum_{j=2}^d v_1 \cdot v_j + 2 \sum_{1<i<j\leq d} v_i \cdot v_j \\
&= 2\Bigl(3 v_1 \cdot v_1 + 2 \sum_{j=2}^d v_1 \cdot v_j + \sum_{1<i \leq j\leq d} v_i \cdot v_j\Bigr) \\
&= 2 \sum_{1\leq i \leq j \leq d} (1+\delta_{1i} + \delta_{1j}) v_i \cdot v_j.
\end{align*}
We get
\[
\e(\norms{VY} Y_k ) = \frac{2}{d(d+1)(d+2)} \sum_{1\leq i \leq j \leq d} (1+\delta_{ki} + \delta_{kj}) v_i \cdot v_j.
\]
The lemma follows.
\end{proof}

\section{Proof of Theorem \ref{thm:main}}

\noindent\emph{Proof idea:} An extremal isotropic simplicial polytope has a hyperplane of symmetry by Theorem \ref{thm:symmetry}. Given that the polytope is simplicial but not a simplex, one can show (Theorem \ref{thm:ccg}, a restatement of results in \cite{MR1763244}) that in this case Blaschke's Sch\"uttelung process strictly increases the isotropic constant.
\begin{proof}
Assume, without loss of generality, that $P$ is in isotropic position. 
By Theorem \ref{thm:symmetry}, $P$ has a hyperplane of symmetry $H$ containing the origin. Let $f_H$, $g_H$ be a representation of $P$ as in \eqref{equ:representation}. 
The hyperplane of symmetry $H$ implies $g_H = f_H$.
By Theorem \ref{thm:ccg}, $f_H$ is linear in $\pi_H(P)$. 
This implies that $\{(x,f_H(x)) \suchthat x \in \pi_H(P) \}$ is a facet of $P$, and therefore a $(d-1)$-dimensional simplex. 
Thus, $\pi_H(P)$ is a $(d-1)$-dimensional simplex.
If $d \leq 2$ the argument below does not work (as being simplicial is no restriction for $d=2$), but the claim of the theorem is obvious for $d=1$ and known to be true for $d=2$, as triangles are the only maximizers of the isotropic constant in that case (see e.g. \cite{MR1763244} for a proof that triangles are maximizers, and \cite{MR2760665} for a proof that they are the only maximizers, which generalizes a result from \cite{MR1203285}).
If $d >2$, the result follows from Lemma \ref{lem:simplex}, with $S = \pi_H(P)$ and $f=f_H$.
\end{proof}

\begin{lemma}\label{lem:simplex}
Let $d \geq 3$. Let $P$ be a $d$-dimensional simplicial polytope of the form $\{(x,y) \in \RR^{d-1} \times \RR \suchthat x \in S, -f(x) \leq y \leq f(x) \}$ where $S \subseteq \RR^{d-1}$ is a $(d-1)$-dimensional simplex with vertices $v_1, \dotsc, v_d \in \RR^{d-1}$ and $f : S \to \RR$ is an affine function. 
Then $P$ is a simplex.
\end{lemma}
\begin{proof}
It is enough to show that $f(v_i)$ is non-zero for exactly one value of $i$. 
We will show this now. 
Let $F_1, \dotsc, F_d$ be the facets of $S$, where $F_i = \conv (\{ v_1, \dotsc, v_d\} \setminus \{v_i\})$.
Let $H= \{ (x,0) \in \RR^{d-1}\times \RR \}$.
For $i \in [d]$, let $H_i$ be the affine hull of $\pi_H^{-1}(F_i)$. 
In other words, $H_i$ is the affine hyperplane orthogonal to $H$ and containing $F_i$.
Now for every $i$, $P \cap H_i$ is a face of $P$. If $\dim P \cap H_i < d-1$, then $f(v_j) = 0$ for all $j$ different from $i$. 
If $\dim P \cap H_i = d-1$, then $P \cap H_i$ is a facet and therefore a simplex, and this implies that $f(v_j)$ is non-zero for at most one value of $j$ in $[d]\setminus\{i\}$. 
In any case, for every set of $d-1$ values in $[d]$, there can be at most one $j$ in it such that $f(v_j) \neq 0$. Assuming $d\geq 3$, this necessarily implies that there is at most one $i \in [d]$ such that $f(v_i) \neq 0$.
\end{proof}

\paragraph{Acknowledgements.} We would like to thank Artem Zvavitch for suggesting to extend the autor's preliminary version of Theorem \ref{thm:symmetry} in 3-D to arbitrary dimensions and Matthieu Fradelizi for suggesting to look at \cite{MR1763244} to find results that would complement Theorem \ref{thm:symmetry}, which lead to Theorem \ref{thm:main}. We would like to thank Matthias Reitzner for helpful discussions.

\bibliographystyle{abbrv}    
\bibliography{isohedron}
\if
\appendix
\section{f}
\[
\frac{d}{dt} \vol K_t = \frac{d}{dt} \int_{-\pi}^{t} \int_0^\infty \vol_{d-2}(S_{\rho,\theta}) \rho d\rho d\theta\\
= \int_0^\infty \vol_{d-2} (S_{\rho, t}) \rho d\rho.
\]
\[
S_{\rho, \theta} = \{ x \in P \suchthat \rho(x) = \rho, \theta(x) = \theta\}
\]

Sectional area of a convex body is logconcave, therefore continuous in int support. Areas of hinging hyperplane? Weighted by $\rho$?

Claim: $\vol K_t$ is differentiable at $t=0$. To see this, note that $\vol K_t = \int_a^t \int_0^\infty \vol_{d-2} (S_{\rho,\theta}) \rho d\rho d\theta$. If we define $a, b \in [-\pi, \pi)$ so that $\theta \in [a,b]$ iff $S_\theta \neq \emptyset$, then $f(\theta) = \int_0^\infty \vol_{d-2} (S_{\rho,\theta}) \rho d\rho$ is continuous in $(a,b)$. As $P$ is strictly larger than $K$, we have $0 \in (a,b)$ and therefore $\vol K_t = \int_a^t f(\theta) d\theta$ is differentiable at $t=0$.

Claim: $\int_{-\pi}^t \int_0^\infty \e_{X \in S_{\rho, \theta}} (\norms{X}) \vol_{d-2} (S_{\rho, \theta}) \rho \,\ud\rho \,\ud\theta$ is differentiable at $t=0$. To see this, observe that $\theta \mapsto \int_0^\infty \e_{X \in S_{\rho, \theta}} (\norms{X}) \vol_{d-2} (S_{\rho, \theta}) \rho \,\ud\rho$ is continuous in $(a,b)$ and $0 \in (a,b)$.

Proof of claim: To see this, note that $g(t) = \int_{-\pi}^t \int_0^\infty \int_{S_{\rho, \theta}} f(\rho, \theta, z) \ud z \rho \ud \rho \ud \theta$.
Let $h(t) = \int_0^\infty \int_{S_{\rho, t}} f(\rho, t, z) \ud z \rho \ud \rho$.
As $P$ is strictly larger than $K$, we have that $h(t)$ is finite and continuous in a neighborhood of $0$ (note that $h(t)$ may not be well defined for large values of $t$ if $P$ is unbounded).

We now show that $h(t)$ is finite in a neighborhood of 0. Note that $h(t)$ is the integral of the continuous function $f(\rho, t, z) \rho$ over $S_t$. The fact that the recession cone of a closed convex set like $P$ is closed implies that if $S_0$ is bounded, the so is $S_t$ for $t$ in a neighborhood of 0. The finiteness of $h(t)$ in that neighborhood follows.

We now show that $h(t)$ is continuous in a neighborhood of 0. This follows from Lebesge's dominated convergence and the continuity of the radial function of $P$ with respect to the origin of the cylindrical system of coordinates in the interior of the tangent cone. 

Therefore, $f(t) = \int_{-\pi}^t h(\theta) d\theta$ is differentiable at $t=0$.

\section{Integration over a simplex}
\begin{align*}
x
\end{align*}
\fi 
\end{document}